\documentclass{article}

\usepackage{arxiv}

\usepackage[hyphens]{url}
\usepackage{subfiles}
\usepackage{lipsum}
\usepackage{graphicx}
\usepackage{ytableau}
\usepackage{todonotes}

\usepackage[utf8]{inputenc}
\usepackage[T1]{fontenc}

\usepackage{comment}
\usepackage{nicefrac}   
\usepackage{amsmath}
\usepackage{bbm}
\usepackage{bm}

\usepackage{amsthm}
\usepackage{newtxtext,newtxmath}
\usepackage{datetime}

\newtheorem{theorem}{Theorem}[section]
\newtheorem{corollary}{Corollary}[section]
\newtheorem{lemma}{Lemma}[section]
\newtheorem{proposition}{Proposition}[section]
\newtheorem{claim}{Claim}[section]
\newtheorem{definition}{Definition}[section]
\newtheorem{example}{Example}[section]
\newcommand{\ps}[1]

\usepackage[
    backend=biber,natbib,
    style=numeric
]{biblatex}
\author{{Yuhui Jin} \\
	Department of Mathematics\\
	California Institute of Technology\\
	Pasadena, CA 91125 \\
	\texttt{yjjin@caltech.edu} \\
	}

\date{}

\renewcommand{\headeright}{}

\DeclareUnicodeCharacter{0308}{HERE!HERE!}
\begin{document}
\title{Central Limit Theorem on the Conjugacy Measure of Symmetric Groups}

\maketitle 
\author{Yuhui Jin}
\begin{abstract}

Regarding the conjugacy representation on symmetric groups, we initiate a normalized measure emerging from this representation, namely the conjugacy measure. A central limit theorem for character ratios of random representations of the symmetric group on the conjugacy measure is obtained.
\end{abstract}
 
\section{Introduction}

Plancherel measure on symmetric groups draws significant interest in the probability perspective. Kerov\cite{Kerov93gaussianlimit}, Ivanov and Olshanski \cite{olshanski} provide a central limit theorem on the character ratio of Plancherel measure of symmetric group. Character ratio is an important quantity to understand random walk on symmetric group generated by transposition.\cite{diaconis111}. Thus, central limit theorem implies spectrum of the random walk is asymptotically normal. A multitude approaches appears in several context. Hora\cite{Hora1998} proves this using method of moment with delicate analysis in random walk on symmetric group. Error bound on this random variable are given in Fulman's paper \cite{fulmansym} using Stein's method. 
q-analog on this measure appears in \cite{F2012},\cite{mellot} and a similar central limit theorem are given.
 
One measure on symmetric groups similar to Plancherel measure is the conjugacy measure. Recall Plancherel measure is defined by the character of irreducible component of regular representation. Instead of regular representation, we take the conjugacy representation, and it yields the conjugacy measure. Adin and Frumkin \cite{Adin1987} realize the ratio of regular and conjugacy character's norm limit to 1. Roichman\cite{Roichman1997} studies the ratio of multiplicity of conjugacy and dimension of the irreducible representation under the constraint of the maximum length of set partition corresponding to the irreducible representation.
In this paper, we are interested a central limit theorem on the character ratio of the conjugacy measure of symmetric groups. In section 2, we first review Hora\cite{Hora1998} and Kerov\cite{Kerov93gaussianlimit}'s theorem on central limit theorem on Plancherel measure. In section 3, we prove a central limit theorem of on the character ratio of the conjugacy measure.

\section{Preliminaries}

\subsection{The Conjugacy Representation of Finite Symmetric Groups}
There are two natural representations of any finite group $G$ on the group algebra $\mathbf{C}[G]=\{\sum_{g\in G} a_gg|a_g\in \mathbb{C}\}$: the (left) regular representation $\phi: \phi(h)\dot \sum_{g\in G} a_gg= \sum_{g\in G} a_ghg$ and the conjugacy representation $\psi: \psi(h)\dot \sum_{g\in G} a_gg= \sum_{g\in G} a_ghgh^{-1}$. 
The following results are from classical representation theory that can be find from any introductory textbook. \[\chi^{\psi(g)}=\frac{|G|}{|C(g)|}\] where $C(g)$ is the conjugacy class of g;\[\chi^{\phi(g)}=|G| \text{ if } g=id \text{ else } \chi^{\phi(g)}=0 \]
In finite symmetric group $S_n$, every partition $\lambda$ of n has an associated irreducible representation $S^\lambda$ of $S_n$ over $\mathbb{C}$. The set $\{S^\lambda|\lambda \text{is a partition of n}\}$ forms the complete set of irreducible representations of $S_n \text{over} \mathrm{C}$. The multiplicity of $S^\lambda$ in a representation $\rho$, denoted by $m(S^\lambda,\rho)$ is given by \[m(S^\lambda,\rho)=\frac{1}{|G|} \sum_{g\in S_n} \chi^{\lambda} (g) \overline{\chi^{\rho}(g)},\]
where $\chi^{\rho}(g)^*$ is the complex conjugate of $\chi^{\rho}(g)$.
Thus $m(S^\lambda, \psi)= \sum_C \chi^{\lambda}(C)$ where the sum runs over all conjugacy class C in G, and $m(S^\lambda, \phi)= \chi^{\lambda}(id).$ 
We rewrite $\chi^{\lambda}(id)$  as $f^\lambda$, as it is the degree of $S^\lambda$. In the world of combinatorics, $f^\lambda$ is also the number of standard Young tableaux of shape $\lambda$. 

\subsection{Notations on Partitions}
Let n be a positive integer. The partitions of n are denoted by $\lambda=(\lambda_1,\lambda_2,\cdots,\lambda_k$ where $\lambda_1\geq \lambda_2\geq\cdots \geq\lambda_k\geq 1$ with $n=\sum_i \lambda_i$.  Let $n(\lambda)$ denotes the number of sum of $\lambda$, which is n in this case. Let $l(\lambda)$ denotes the order of $\lambda$, which is k in this case.  Given $\lambda\vdash n $ $l_{i}(\lambda)=\text{the number of j appearing }\lambda$. Next, we define two products involving $l_i(\lambda)$. We define $c(\lambda)=\prod_i (l_i(\lambda)! i^{l_i(\lambda)})$, where $\frac{n(\lambda)}{c(\lambda)}$ is the size of conjugacy class of $\lambda$ in $S_{n(\lambda)}.$ We write $d(\lambda)=\prod_i l_i(\lambda)! $
\subsection{Plancherel Measure}
\begin{definition}
\label{def1}
Let n be a positive integer. We denote the set of its irreducible representations by 
$S_n^{\wedge}$. Specifically, this is equivalent to the set consisting of all integer partitions of n.  The corresponding Plancherel measure over the set 
$S_n^{\wedge}$ is defined by $\mu_\phi(\lambda)=\frac{(f^{\lambda})^2}{n!}$.
\end{definition}

\begin{definition}
Similarly to definition \ref{def1}, we define a measure $\mu_\psi(\lambda)=\frac{f^{\lambda}\sum_{K\in \hat S_n} \chi^{\lambda}(K)}{|S_n|}$ to be the conjugacy measure. 
\end{definition}

Recall in Subsection 2.1, we showed that the multiplicity of the regular representation and the conjugacy representation. By replacing one $f^\lambda$ in the Plancherel measure by $m(S^\lambda,\psi)$, we have the definition of the conjugacy measure. 
The following equality and Corollary are classical results from representation theory, which shows that the conjugacy measure is well defined. 
\[\sum_{\lambda} \mu_\psi(\lambda)=1 \]
\begin{corollary}
    For any finite group G and any irreducible representation $\rho$ of G, the sum $\sum_C \chi^{\rho}(C)$ is a nonnegative integer.
\end{corollary}

\begin{theorem}
\label{cltplancherel}
Take the Plancherel measure $\mu_{\phi}$ on the irreducible representation of the symmetric groups $S_n$,which is, in our case, on $\lambda\vdash n$ .The random variable \[X_{n,\nu}(\lambda)= \frac { [n]_{n(\nu)} \times \chi^{\lambda}(\nu)\sqrt{d(\nu)} }{f^\lambda \times n^{n(\nu)/2}\sqrt{c(\nu)} }\]is asymptotically normal with mean 0 and variance 1 if $\nu$ satisfies the condition $l(\nu)=1$ and $n(\nu)>1$. For any $\nu^{(1)},\nu^{(2)},\cdots, \nu^{(m)}$ partitions with any $n(\nu^{(i)})<n$ and $l_1(\nu^{(i)})=0$ and any $k_1,k_2, \cdots k_m\in \mathrm{N}$, we have
\begin{equation}
    \lim_{n\to\infty}E[X_{n,\nu^{(1)}}^{k_1} X_{n,\nu^{(2)}}^{k_2}\cdots X_{n,\nu^{(m)}}^{k_m}]=\prod_{j\geq 2} \int_{-\infty}^\infty \frac{e^{-x^2}}{\sqrt{2\pi}} (\mathcal{H}_{l_j(\nu^{(1)})}(x))^{k_1} (\mathcal{H}_{l_j(\nu^{(2)})}(x))^{k_2} \cdots (\mathcal{H}_{l_j(\nu^{(m)})}(x))^{k_m} dx
\end{equation}
where $\mathcal{H}$ is the Hermite polynomial. 
\end{theorem}
Theorem \ref{cltplancherel} is a classical result in Kerov\cite{Kerov93gaussianlimit} and Hora\cite{Hora1998}. Both of these proofs establish asymptotic normality by the
method of moments and use combinatorial methods to estimate the moments. Hora first proves the second half of the theorem.

\section{Conjugacy Measure}

\begin{theorem}
\label{clt_main}
Take the conjugacy measure $\mu_{\psi}$ on the irreducible representation of the symmetric group $S_n$, which is on $\lambda\vdash n$. Given any $\nu$ with $n(\nu)<n$ and $l_1(\nu)=0$, the random variable 
\begin{equation}
W_{n,\nu}(\lambda)= \frac { [n]_{n(\nu)} \times \chi^{\lambda}(\nu)\sqrt{d(\nu)}}{f^\lambda \times n^{n(\nu)/2}\sqrt{c(\nu)} }
\end{equation}
is asymptotically normal with mean 0 and variance 1 if $\nu$ satisfies the condition $l(\nu)=1$ and $n(\nu)>1$. For any $\nu^{(1)},\nu^{(2)},\cdots, \nu^{(m)}$ partitions with any $n(\nu^{(i)})<n$ and $l_1(\nu^{(i)})=0$ and any $k_1,k_2, \cdots k_m\in \mathrm{N}$, we have
\begin{equation}
    \lim_{n\to\infty}E[W_{n,\nu^{(1)}}^{k_1} W_{n,\nu^{(2)}}^{k_2}\cdots W_{n,\nu^{(m)}}^{k_m}]=\prod_{j\geq 2} \int_{-\infty}^\infty \frac{e^{-x^2}}{\sqrt{2\pi}} (\mathcal{H}_{l_j(\nu^{(1)})}(x))^{k_1} (\mathcal{H}_{l_j(\nu^{(2)})}(x))^{k_2} \cdots (\mathcal{H}_{l_j(\nu^{(m)})}(x))^{k_m} dx
\end{equation}
\end{theorem}

Before starting the proof of Theorem \ref{clt_main}, we first provide new definitions and two results that are useful in the proof.

\begin{definition}
Let  $\omega=(\omega_1,\omega_2,...\omega_k)$ be a k-tuple with $\omega_i$ as a cycle notation in $S_n$. For any non-negative integer j, let $b_j$ denote the number of times j appears in $\omega$. Arrange the counts $\{b_j\}'s $ in decreasing order and remove any trailing zeros. The resulting sequence is the symbol of $\omega$, denoted as $\overline{\omega}$. The product of $\omega, p(\omega)$ denotes the partition that corresponds to the cycle type of product of $\omega_i$'s. Here is an example of $p(\omega)$. 
\end{definition}

\begin{example}
For a transposition 4-tuple, $((1,2)(1,3)(5,6)(1,6))$ in $S_8$, the symbol is (3,2,1,1). The product is $(12365)$. $p(\omega)=\{5,1,1,1\}$.
\end{example}

\begin{lemma}
\label{keylemma}
(\cite{stanley} Ex 7.67) Let G be a finite group with conjugacy classes $C_1,C_2,..C_r$. Let $C_k$ be the conjugacy class of an element $\omega\in G$ Then the number of m-tuples $(q_1,q_2,..q_m)\in G^m$such that $g_j\in C_{i_j}$ and $g_1...g_m=\omega$ is \[\frac{\prod_{j=1}^m |C_{i_j}|}{|G|} \sum_{\chi\in Irr(G)} \frac{1}{dim(\chi)^{m-1}} \chi(C_{i_1})\cdots\chi(C_{i_m})\overline{\chi(C_k)}\]
\end{lemma}

\begin{definition}
 Let $\omega=(\omega_1,\omega_2,...\omega_k)$ be a k-tuple with $\omega_i$ as a cycle notation in $S_n$. For each $\omega_i$, let $\omega_i^{1}\omega_i^{2}\cdots \omega_i^{r_i}$ be the cycle notation of $\omega_i$. Assume that $\overline{\omega}=\{q,2,\cdots, 2,1\cdots,1\}$ with $q\geq 2$, 
 we denote the element that corresponds to $q$ as a i(if the maximum of $\overline{\omega}$ is 2, we pick any element that appears twice as a). We write $\omega_{i_{1}}^{j_1},\cdots \omega_{i_{q}}^{j_q}$ with $i_i<i_2<\cdots<i<q$  as cycles containing element $a$, and for each cycle, we rearrange them in the form that the element $a$ is in the first position. 
 We denote $\omega_{i}^j[p]$ as the p-th element in the cycle $\omega_{i}^j$ and specifically $\omega_{i}^j[-1]$ as the last element. If there exists t such that $\omega_{i_t}^{j_t}[-1]=\omega_{i_{t+1}}^{j_{t+1}}[2]$ or $\omega_{i_t}^{j_t}[2]=\omega_{i_{t+1}}^{j_{t+1}}[-1]$, we delete the two equal element from the tuple. After deletion, if some cycle turns to a singleton, we remove this cycle from the tuple. In this case, when $t=q$, we take $\omega_{i_{q+1}}^{j_{q+1}}$ as $\omega_{i_1}^{j_1}.$  We repeat this procedure until either the symbol of current tuple does not have element greater than one or there is no deletion available. We call a tuple after the entire procedure reduced tuple.\label{algorithm}
\end{definition}
\begin{proposition}
 Let $\omega=(\omega_1,\omega_2,...\omega_k)$ be a k-tuple with each $\omega_i$ as a cycle notation in $S_n$. Assume that $\overline{\omega}=\{i,2,\cdots, 2,1,\cdots,1\}$ with $i\geq 2$, we have $p(\omega)=id$ if and only if there exists a reduced tuple $\tilde{\omega}$ of $\omega$ is $\emptyset$.
\label{algprop}
\end{proposition}

\begin{proof}
 Notice that for each pair of elements deleted by $\omega_{i_t}^{j_t}[-1]=\omega_{i_{t+1}}^{j_{t+1}}[2]$ or $\omega_{i_t}^{j_t}[2]=\omega_{i_{t+1}}^{j_{t+1}}[-1]$ does not appear in other places in $\omega$. The procedure in Definition $\ref{algorithm}$ preserves the product of $\omega$ and the symbol $\{q',2,2,\cdots,2,1\cdots 1\}$ with $q-q'geq 0$. Thus, if a reduced tuple of $\omega$ is $\emptyset$, the product $p(\omega)=id$.

 Next, we prove if any reduced tuple of $\omega$ is nonempty and the symbol of $\omega$ is not identity, there is a contradiction. Without loss of generality, we take a $\tilde{\omega}$ as $\omega$. Therefore, all notations in Definition $\ref{algorithm}$ follow in this proof. Denote $\omega =L_1\omega_{i_{1}}^{j_1}L_2 \omega_{i_{2}}^{j_2} L_3 \cdots L_q\omega_{i_{q}}^{j_q}L_{q+1}.$ Since $p(\omega)=id$, we have $l_1(\overline{\omega}=0.$ We rewrite $\omega =\omega_{i_{1}}^{j_1}L_2 \omega_{i_{2}}^{j_2} L_3 \cdots L_q\omega_{i_{q}}^{j_q}L_{q+1}L_{1}^{-1},$ with $p(\omega)=id$. Rewrite $L_{q+1}=L_{q+1}L_{1}^{-1}.$ Denote $\omega_{i_{t}}^{j_t}=(a,b_{t,0},b_{t,1},\cdots,b_{t,y_t})$. 
The element $b_{1,y_1}$ goes to $a$ in $\omega_{i_{1}}^{j_1}$, and needs to be mapped back to itself after $L_{q}$. If the element $b_{1,y_1}$ appears in $L_1$, by the fact $b_{0,y_0}$ only appearing twice,  it cannot be mapped back after $L_{q}$. Since $\omega$ is already of the reduced form, even if $b_{1,y_1}$ is in $\omega_{i_{2}}^{j_2}$, it cannot be mapped back after $L_{q}$. Thus, $b_{1,y_1}$ appears in $L_2\omega_{i_3}^{j_3}\cdots L_{q}$.
The element $a$ goes to $b_{1,0}$ in $\omega_{i_{1}}^{j_1}$, and needs to be mapped back to $a$ after $L_{q}$. If the element $b_{1,0}$ appears in $L_q$, by the fact it only appearing twice, it cannot be mapped back after $L_{q}$. Since $\omega$ is already of the reduced form, even if $b_{1,0}$ is in $\omega_{i_{q}}^{j_q}$, it cannot be mapped back to $a$ after $L_{q}$. Thus, $b_{1,0}$ appears in $L_1\omega_{i_2}^{j_2}\cdots L_{q-1}$.
We call an element $g$ in $\omega_{i}^{j}$ appears before $f$ in $\omega_{i'}^{j'}$  if $i\leq i'\text{ and } j\leq j'$.  
The element $b_{1,0}$ goes to $b_{1,1}$ in $\omega_{i_{1}}^{j_1}$, and thus $b_{1,1}$ appears before the second $b_{1,0}$. Therefore, all elements $b_{1,i}$ appear in $L_1\omega_{i_2}^{j_2}\cdots L_{q-1}$, with the order of $b_{1,y_1}, b_{1,y_1-1},\cdots, b_{1,0}$. Recall $b_{1,y_1}$ appearing in $L_2\omega_{i_3}^{j_3}\cdots L_{q}$, all $b_{1,i}$ appear in  $L_2\omega_{i_3}^{j_3}\cdots L_{q-1}.$

Revisiting $b_{1,y_1}$, we have $b_{1,y_1}$ goes to $a$ in $\omega_{i_{1}}^{j_1}$, and then goes to  $b_{2,0}$ in $\omega_{i_{2}}^{j_2}.$ Similar to the argument for all elements $b_{2,i}$ appear in $L_1\omega_{i_2}^{j_2}\cdots L_{q-1}$ ($b_{1,y_1}$ going to $a$ then to $b_{2,0}$ needs to be back to $b_{1,y_1}$), $b_{2,0}$ is in $L_2\omega_{i_3}^{j_3}\cdots L_{q-1}$. Thus, all $b_{2,i}$ appears in  $L_2\omega_{i_3}^{j_3}\cdots L_{q-1}.$ For $b_{2,y_2}$, similar to $b_{1,y_1}$ in $L_2\omega_{i_3}^{j_3}\cdots L_{q}$ ($b_{1,y_1}$ cannot be mapped to itself if it's in $L_1$ or $\omega_{i_{2}}^{j_2}$), $b_{2,y_2}$ appears in $L_3\omega_{i_4}^{j_4}\cdots L_{q}$. Now all $b_{2,i}$ appear in  $L_3\omega_{i_4}^{j_4}\cdots L_{q-1}.$

Redo this analysis until $b_{q-1,0}.$ We have all $b_{q-1,i}$ appearing in $L_{q-1}$
$b_{q-1,y_{q-1}}$ maps to $a$ and cannot be mapped back to $b_{q-1,y_{q-1}}$ (it is in $L_{q-1}$). We reach a contradiction.
 
\end{proof}
\begin{lemma}
Assume that $\overline{\omega}=\{2^q\}$, from Proposition \ref{algprop}, the product of $\omega$ is identity if and only given any $i,j$,  there exists a unique $i',j'$ such that  $i\neq i'$with $\omega_{i}^j\omega_{i'}^{j'}=id$.\label{keyobserv}
\end{lemma}

\begin{proof}
This Lemma was first proven in \cite{Hora1998}. In this paper, we give an alternative approach, with the help of Proposition \ref{algprop}. Since $p(\omega)=id$, a reduced tuple of $\omega$ is the empty tuple. Therefore, we retrieve the original $\omega$ from the empty tuple by following the procedure backward. Since for any $\omega$ with $p(\omega)=id$, the symbol of $\omega$ does not have component of one ($l_1(\overline{{\omega}})=0$). Therefore, for $\omega$ with symbol comprised of 2, for each procedure, deleted elements must be in pair. Tuples that lead to a empty tuple one step before are of the shape$((a,b)(a,b)).$  Then we present a proof by induction on the number of steps before empty set: Tuples that leads to empty tuples n steps before have the property that each cycle's inverse exists uniquely thus forming a pair in the tuple. Take a tuple $\omega=(\omega_1,\omega_2,...\omega_q)$, for each $i$, there exists a unique $j$ such that $i\neq j$ with $\omega_i=\omega_j^{-1}.$ We assume the statement for step n works. For step n+1, given tuple $\omega=(\omega_1,\omega_2,...\omega_q)$, it can be modified as follows: 1. Pick any element $q$ appearing in $\omega$ and insert a new element $t$ in the two cycles containing $q$ under the rule of the procedure. 2. Pick two distinct elements $q,t$ not in $\omega$ and add $(q,t)$ into $\omega$ twice. We notice that the second method of changing tuple still preserve the inductive statement. For the first method, we write the two cycles containing $q$ as $\omega_i,\omega_j (i<j)$ with the first element of both being $q.$ Then inserting $t$ either into the second position of $\omega_i$ and the last position of $\omega_j$ or vice versa preserves the relation$\omega_i=\omega_j^{-1}$ on the updated $\omega_i,\omega_j$.  Therefore, any $\omega$ with a empty reduced tuple has the the pairing property $\omega_i=\omega_j^{-1}.$

For the other direction, a first observation is the number of cycles in $\omega$ is even. Given $\omega$ of the above property, with loss of  generosity, $\omega_1=\omega_j^{-1},\omega_1\omega_2...\omega_{2k}=\omega_1(\omega_2)\times\omega_1(\omega_3)\times...\times\omega_1(\omega_{j-1})\times \omega_{j+1}\times..\omega_{2k}$.It is due to a fundamental result in symmetric group: $\sigma(a_1,...a_i)\sigma^{-1}=(\sigma(a_1),\sigma(a_2),...\sigma(a_i))$. Note $\overline{\omega}=\{2,,,2\}$, each element in $\omega_{i},\text{ except for } i= 1,j, $ does not appear in $\omega_1$. Thus each $\omega_i$ is fixed under the $\omega_1$. Explicitly, we notice $\omega_1\omega_2...\omega_{2k}=\omega_2\omega_3...\omega_{j-1}\omega_{j+1}...\omega_{2k}$. Redo this for k times shows that product of $\omega$ is the identity.

\end{proof}
\begin{proof}
By the statement of method of moment, for the first half of the Theorem, it suffice to show that moments of any order for W is finite and moments converge to the respective moments of Gaussian variable. Explicitly, \[\lim_{n\to\infty}E[W_{n,\nu}^{2k}]=(2k-1)!!\] and \[\lim_{n\to\infty}E[W_{n,\nu}^{2k-1}]=0\] for any k,$n(\nu)\in \mathbb{Z} \text{ and } k,n(\nu)>0$. Therefore, the second half statement of the Theorem is a generalization of the first half. 

\begin{equation}
E[W_{n,\nu^{(1)}}^{k_1} W_{n,\nu^{(2)}}^{k_2}\cdots W_{n,\nu^{(m)}}^{k_m}]=\sum_{\lambda\vdash n} \frac{f^{\lambda} \sum_{\delta \vdash n}\chi^\lambda(\delta)}{n!} (\prod_{i}(\frac { \sqrt{d(\nu^{(i)})}[n]_{n(\nu^{(i)}} \chi^{\lambda}(\nu^{(i)})}{f^\lambda n^{n(\nu^{(i)})/2}\sqrt{c(\nu^{(i)})} })^{k_i} ) \label{l1}  
\end{equation}
To avoid heavy notation on the right hand side of Equation \eqref{l1}, we write $B_{n,\nu,k,K}$ to be 
\begin{equation}
B_{n,\nu,k,\delta}=\frac{\prod_i ([n]_{n(\nu^{(i)}}/c(\nu^{(i)}))^{k_i}}{n!}\sum_{\lambda\vdash n} \frac{\prod_i \chi^\lambda(\nu^{(i)})^{k_i}}{(f^\lambda)^{k-1}}  \chi^\lambda (\delta) \frac{n!}{c(\delta)}.  \label{l2} 
\end{equation}
Thus, we rewrite Equation \eqref{l1} in a simpler form.
\begin{equation}
\eqref{l1}=\sum_{\delta\vdash n} \frac{c(\delta)}{n!} B_{n,\nu,k,\delta}\prod_i (\frac{d(\nu^{(i)})c(\nu^{(i)})}{n^{n(\nu^{(i)})}})^{k_i/2}\label{l3} 
\end{equation}

Recall that $\frac{n!}{c(\delta)}$ is the size of the conjugacy class of $\delta$. Given an element $g\in S_n$ with g in the conjugacy class of $\delta$, applying Lemma \ref{keylemma} to \eqref{l2} shows that $\frac{B_{n,\nu,k,\delta}}{n!/c(\delta)}$ is the number of k-tuples $(g_1^{(1)},g_2^{(1)},\cdots,g_{k_1}^{(1)}\cdots g_1^{(m)},g_2^{(m)},\cdots,g_{k_m}^{(m)})$ where each $g_j^{(i)}$ is in the conjugacy class of $\nu^{(i)}$, such that $g_1^{(1)}...g_{k_m}^{(m)}=g$. Thus, $B_{n,\nu,k,\delta}$ is the number of k-tuples $(g_1^{(1)},g_2^{(1)},\cdots,g_{k_1}^{(1)}\cdots g_1^{(m)},g_2^{(m)},\cdots,g_{k_m}^{(m)})$ where each $g_j^{(i)}$ is in the conjugacy class of $\nu^{(i)}$, such that $g_1^{(1)}...g_{k_m}^{(m)}$ is in the conjugacy class of $\delta$. We denote the set consisting of such tuples by $\mathcal{T}_{n,\nu,k,\delta}.$ In this proof, we slightly abuse the notation on $\nu$ and $k$ by writing $n(\nu)=\sum_i n(\nu^{(i)})$, $k=\sum_i k_i$ and $n(\nu)k=\sum_i  n(\nu^{(i)})k_i.$ 

It suffice to show the following:
\begin{itemize}
    \item \begin{equation}
    \lim_{n\to \infty} \sum_{\delta\neq \emptyset,\delta\in \hat S_n} \frac{c(\delta)}{n!} B_{n,\nu,k,\delta}\prod_i (\frac{d(\nu^{(i)})c(\nu^{(i)})}{n^{n(\nu^{(i)})}})^{k_i/2}=0;  \label{ll1}
    \end{equation}
    \item  \begin{equation}\lim_{n\to \infty}  B_{n,\nu,k,\{1^n\}}(\frac{c(\nu)}{n^{n(\nu)}})^{k/2}=0 \label{ll2}\end{equation} if $k$ is odd and $\nu=\{q\}$ ;     
    \item  \begin{equation}\lim_{n\to \infty}  B_{n,\nu,k,\{1^n\}}(\frac{c(\nu)}{n^{n(\nu)}})^{k/2}=(2k-1)!!\label{ll3}\end{equation} if $k$ is even and $\nu=\{q\}$ ;  
    \item \begin{equation}\lim_{n\to \infty}  B_{n,\nu,k,\delta}\prod_i (\frac{d(\nu^{(i)})c(\nu^{(i)})}{n^{n(\nu^{(i)})}})^{k_i/2}=\prod_{j\geq 2} \int_{-\infty}^\infty \frac{e^{-x^2}}{\sqrt{2\pi}} (\mathcal{H}_{l_j(\nu^{(1)})}(x))^{k_1} (\mathcal{H}_{l_j(\nu^{(2)})}(x))^{k_2} \cdots (\mathcal{H}_{l_j(\nu^{(m)})}(x))^{k_m} dx \label{ll4}\end{equation}
\end{itemize}

We first prove $\eqref{ll1}$. Therefore, in this part, every $\delta$ mentioned has the property that it is not comprised of one's.

\[ \sum_{\delta\neq \{1^l\},\delta\in \hat S_n} B_{n,\nu,k,\delta} (\prod_i (\frac{d(\nu^{(i)})c(\nu^{(i)})}{n^{n(\nu^{(i)})}})^{k_i/2}) \frac{c(\delta)}{n!}= \sum_{\omega \in \mathcal{T}_{n,\nu,k,\delta},\delta\vdash n, \delta\neq \{1^n\}} \frac{c(\delta)}{n!} \prod_i (\frac{d(\nu^{(i)})c(\nu^{(i)})}{n^{n(\nu^{(i)})}})^{k_i/2}\]

Rearranging and grouping the sums in all tuples in $\mathcal{T}_{n,\nu,k,\delta}$ for all $\delta$ provides the following: 
\begin{equation}
=\sum_{\mu\vdash n(\nu)k,\delta \vdash n } {n \choose {l_1(\mu),l_2(\mu),...l_i(\mu)}}(\prod_i (\frac{d(\nu^{(i)})c(\nu^{(i)})}{n^{n(\nu^{(i)})}})^{k_i/2})\frac{c(\delta)}{n!}C_{\mu,\delta,\nu,k},\label{l4}
\end{equation}

where $C_{\mu,\delta,\nu,k}=|\{\omega\in \mathcal{T}_{l(\mu),\nu,k,\delta}|\overline{\omega}=\mu,p(\omega)=\delta\}|$.  Equality \eqref{l4} sums over $\mu, \delta$. We notice that all tuples in $\mathcal{T}_{l(\mu),\nu,k,\delta}$ take elements from $S_{l(\mu)}$, and we slightly abuse the notation of cycle type equality by ordering $\delta, \eta$ is equivalent if the only difference is the number of trailing one's. To be explicit, $\{4,3,1,1,1\}$ is equivalent to $\{4,3,1\}$
 
Observing that given any $\omega\in \mathcal{T}_{n,\nu,k,\delta},$  the inequality $l_1(p(\omega))\geq n-n(\nu)k$ always holds since any tuple only contains $n(\nu)k$ elements and thus at least $n-n(\nu)k$ fixed elements. Instead of summing over all $\delta$, we only need to sum over partitions of n with at least $n-n(\nu)k$'s trailing ones, since all other $\delta$ correspond to $C_{\mu,\delta,\nu,k}=0$.

\[\eqref{l4}=\sum_{\mu,\delta \vdash n(\nu)k} {[n]_{\sum_{l_i(\mu)}}}(\prod_i (\frac{d(\nu^{(i)})c(\nu^{(i)})}{n^{n(\nu^{(i)})}})^{k_i/2}) \frac{l_{1}(\delta+1^{n-n(\nu)k})!}{n! } C_{\mu,\delta,\nu,k}\frac{\prod_{j>1}l_{j}(\delta)! j^{l_{j}(\delta)}}{\prod {l_j(\mu)}!}\]
Recall that $\delta$ does not consist of one's, for any nonzero $C_{\mu,\delta,\nu,k}$, the corresponding $\delta,\mu$ has inequality \begin{equation}l_1(\delta)+\sum_i l_i(\mu)\geq n(\nu)k.\label{l5}\end{equation} For any tuple $\omega$ in the set corresponding $C_{\mu,\delta,\nu,k}$, each $i \in S_{l(\mu)}$ is either fixed by the product of $\omega$, or it appears in some places in $\omega$. Since $\delta,\mu\vdash n(\nu)k$, the number of fixed elements by $\omega$ in $S_{n(\nu)k}$ is $l_1(\delta)$ and the number of elements appearing in $\omega$ is $\sum_i l_i(\mu)$. With the inequality \eqref{l5}, we provide an equality needed for the next step:
\[
[n]_{l_1(\delta)-n(\nu)k+\sum_{j} {l_j(\mu)}}  \frac{[n-\sum_{j} l_j(\mu)]_{n(\nu)k-l_1(\delta)}}{[n]_{n(\nu)k-l_1(\delta)}}=[n]_{\sum_{l_i(\mu)}}\frac{l_{1}(\delta+1^{n-n(\nu)k})!}{n!} 
\]

By rearranging order of terms, we have 
\begin{equation}\eqref{l4}=\sum_{\mu,\delta \vdash n(\nu)k} (\frac{[n]_{l_1(\delta)-n(\nu)k+\sum_{j} {l_j(\mu)}}}{\prod_i n^{n(\nu^{(i)})k_i/2}})\frac{C_{\mu,\delta,\nu,k} \prod_i (d(\nu^{(i)})c(\nu^{(i)}))^{k_i/2} \prod_{j>1}l_{j}(\delta)! j^{l_{j}(\delta)} }{\prod_j {l_j(\mu)}!} \frac{[n-\sum_{j} l_j(\mu)]_{n(\nu)k-l_1(\delta)}}{[n]_{n(\nu)k-l_1(\delta)}}.\label{l10}
\end{equation}

We denote \begin{equation}
    A_{\mu,\delta}=\frac{C_{\mu,\delta,\nu,k}\prod_i (d(\nu^{(i)})c(\nu^{(i)}))^{k_i/2}\prod_{j>1}l_{j}(\delta)! j^{l_{j}(\delta)} }{\prod_j {l_j(\mu)}!} \frac{[n-\sum_{j} l_j(\mu)]_{n(\nu)k-l_1(\delta)}}{[n]_{n(\nu)k-l_1(\delta)}}.
\end{equation}
    Thus,\[\eqref{l4}=\sum_{\mu,\delta \vdash n(\nu)k} (\frac{[n]_{l_1(\delta)-n(\nu)k+\sum_{j} {l_j(\mu)}}}{\prod_i n^{n(\nu^{(i)})k_i/2}}) A_{\mu,\delta}\] 
Since there is no n in $C_{\mu,\delta,\nu,k}$, we notice that $C_{\mu,\delta,\nu,k}=O(1)$.
By
\[\lim_{n\to\infty} \frac{(n-\sum_{j} l_j(\mu))_{n(\nu)k-l_1(\lambda)}}{(n)_{n(\nu)k-l_1(\lambda)}}=1,\] we obtain $A_{\delta,\mu}=O(1)$.

\[\eqref{l4}\leq \sum_{\mu,\delta \vdash n(\nu)k} n^{l_1(\delta)-n(\nu)k+\sum_{j} {l_j(\mu)}-n(\nu)k/2}A_{\mu,\delta}\] 

Note that for any $\mu,\delta \text{ such that } C_{\mu,\delta,\nu,k}\neq 0$, pick any $\omega\in\mathcal{T}_{l(\mu),\nu,k,\delta}\text{ where }\overline{\omega}=\mu,p(\omega)=\delta $. Elements in the underlying set of $S_{n(\nu)k)}$ that only appear once in $\omega$ can't be fixed points in the conjugacy class of $\delta$. Explicitly, \begin{equation}
n(\nu)k-l_1(\mu)\geq l_1 (\delta)\label{l6},\end{equation} with the equality holds when all elements appearing more than once are fixed points. Another inequality needed is, \begin{equation}\sum_{j>1} l_{j}(\mu)\leq n(\nu)k/2\label{l7}\end{equation} with equality holds when $\mu=(2,...,2)$. When $\mu\vdash n(\nu)k$ and $\mu$ comprises of 2's, the inequality \eqref{l6} is a strict inequality since $l_1(\delta)<n(\nu)q$, ($\delta$ does not comprise of 1), and $l_1(\mu)=0$. By inequalities $\eqref{l6},\eqref{l7}$, we obtain $l_1(\delta)-n(\nu)k+l_1(\mu)+\sum_{j>1} {l_j(\mu)}-n(\nu)k/2\leq 0.$ Since the first four summands of the left side inequality are integer, by  equality condition on $\eqref{l6},\eqref{l7}$ we rewrite it as $l_1(\delta)-3n(\nu)k/2+l_1(\mu)+\sum_{j>1} {l_j(\mu)}\leq -1/2.$

Therefore, \[\eqref{l4}\leq \sum_{\mu,\delta \vdash n(\nu)k} n^{-1/2}A_{\mu,\delta}=\sum_{\mu,\delta \vdash n(\nu)k,A_{\mu,\delta}\neq 0 }O(n^{-\frac{1}{2}})\]
The sum over $\mu,\delta$ which are partition of $n(\nu)k$ are constant regarding n, so when n$\to\infty$,$\eqref{l4}\to$0. The inequality $\eqref{ll1}$ is proven.

Equalities $\eqref{ll2}, \eqref{ll3}$ and $\eqref{ll4}$ was first proven in \cite{Hora1998}. In this paper, since we applied combinatorial computation a slightly different set, we will give a detailed proof of $\eqref{ll2}$ and $\eqref{ll3}$. And we refer readers to \cite{Hora1998} of a full proof about Equality $\eqref{ll4}$.

For  equalities $\eqref{ll2}$ and $\eqref{ll3}$, we start by modify the proof of inequality $\eqref{ll1}$. Every step before inequality $\eqref{l10}$ works for the inequalities $\eqref{ll2}$ and $\eqref{ll3}$, since the condition on $\delta\neq\{1^{n(\nu)k}\}$ is not applied and we have not yet sum over all $\delta.$

\begin{equation} B_{n,\{q\},k,\{1^{qk}\}}(\frac{c(\{q\})}{n^{q}})^{k/2}=\sum_{\mu \vdash qk} (\frac{1}{n^{q}})^{k/2}{[n]_{\sum_{j} {l_j(\mu)}}}\frac{C_{\mu,\{1^{qk}\},\{q\},k}c(\{q\})^{k/2} }{\prod_j {l_j(\{q\})}!}\label{l11}
\end{equation}

Each element of any tuple’s symbol in $\mathcal{T}_{n,\{q\},k,\{1^{qk}\}}$ is greater than 1 since otherwise the corresponding element in the tuple is not a fixed point. That is, $l_1(\mu)=0$ for nonzero $C_{\mu,\{1^{qk}\},\{q\},k}.$
Since $\sum_{j}l_j(\mu)j=qk, l_1(\mu)=0,\text{ then } \sum_{j>1}l_j(\mu)\leq qk/2$ with equality holds when  $\mu=\{2,2,....2\},k\text{ or } q \text { is even}.$
The above limit is 0 unless the equality holds. This boils down to \begin{equation}\lim_{n\to\infty}  B_{n,\{q\},k,\{1^{qk}\}}(\frac{c(\{q\})}{n^{q}})^{k/2}=\frac{C_{\mu,\{1^{qk}\},\{q\},k}c(\{q\})^{k/2} }{(qk/2)!}.\label{l15}\end{equation}

Under $\mu=\{2,2,....2\}$, from Lemma $\ref{keyobserv}$, the total length of any $\omega$ in the set corresponding to $C_{\mu,\{1^{qk}\},\{q\},k}$ is even ($k$ is even). Since the $p(\omega)$ is the partition comprised of 1, $k(q-1)$ is even. The discussion above indicates $k$ is even given $C_{\mu,\{1^{qk}\},\{q\},k}\neq 0$. The equality $\eqref{ll2}$ is proven.

Similarly as the proof for the equality $\eqref{ll1}$, we write \begin{equation}
    \eqref{l15}=\frac{C_{\{2^{qk/2}\},\{1^{qk}\},\{q\},k}q^{k/2}}{qk/2!}.
\end{equation}

We write our original k as 2k for the rest of the proof for equality \eqref{ll3}.
Thus, with Lemma \ref{keyobserv}, the computation boils down to a pure enumerative combinatorics problem. $\frac{(qk)!}{q^k\times k!}$ is the number of k q-cycle with each element distinct and comes from $\{1,2,...qk\}$. $\frac{(2k)!}{2^k}$ it the number of ways assigning k elements in 2k ordered block with each element appears exactly twice. 

For equality \eqref{ll4}, we modify \eqref{l10} and rewrite \eqref{ll4} as \[\eqref{ll4}=\lim_{n\to \infty}\sum_{\mu \vdash n(\nu)k} (\frac{[n]_{\sum_{j} {l_j(\mu)}}}{n^{n(\nu)k/2}}) \frac{C_{\mu,id,\nu,k} \prod_i (d(\nu^{(i)})c(\nu^{(i)}))^{k_i/2 }}{\prod_j {l_j(\mu)}!}. \]
Similarly, the only nonzero term after taking n to infinity is $\mu=\{2,2,\cdots,2\}.$
\begin{equation}\eqref{ll4}=\frac{C_{\{2^{n(\nu)k/2}\},id,\nu,k} \prod_i (d(\nu^{(i)})c(\nu^{(i)}))^{k_i/2 }}{ n(\nu)k/2!}. 
\end{equation}
Again, with Lemma \ref{keyobserv}, the computation boils down to a pure enumerative combinatorics problem.

A detailed computation is given in \cite{Hora1998}. The general strategy is by the theory of matching polynomials of graphs and it's connection with Hermite polynomials. Chapter 1 of \cite{godsil2017algebraic} provides a detailed introduction to this topic.

\end{proof}

\section{Further Direction}
One interesting further step is to investigate whether the error bound on the moment will imply any error bound on the random variable. Johansson \cite{Johansson} provides analysis that reverse engineering this process. 
Another interesting direction is to generalize this method of moment to the q-analog of Plancherel measure or conjugacy measure. The case in Plancherel measure appears naturally in Iwahori-Hecke algebra representation. Also note that Murnaghan-Nakayama rule in q-analog provides a recursive combinatorial method of computing irreducible character of Hecke-algebra. This is provided in Remmel and Ram's work \cite{Ram1997}. 
Similar to the method of mement  in q-analog of Plancherel measure, it is interesting to investigate  a similar q-analog of the conjugacy measure, and potentially a central limit theorem in it's q-analog version.

\section{acknoledgements}
The author thanks Jason Fulman for initiating this question and useful conversation. The author also thanks Eric Rains for useful comments.
\printbibliography[heading=bibintoc]

@article{Roichman1997,
  title={Decomposition of the conjugacy representation of the symmetric groups},
  author={Roichman, Yuval},
  journal={Israel Journal of Mathematics},
  volume={97},
  pages={305--316},
  year={1997},
  publisher={Springer}
}

@article{Johansson,
  title={On random matrices from the compact classical groups},
  author={Johansson, Kurt},
  journal={Annals of mathematics},
  pages={519--545},
  year={1997},
  publisher={JSTOR}
}

@article{F2012,
  title={Asymptotics of q-Plancherel measures},
  author={F{\'e}ray, Valentin and M{\'e}liot, Pierre-Lo{\"\i}c},
  journal={Probability Theory and Related Fields},
  volume={152},
  pages={589--624},
  year={2012},
  publisher={Springer}
}

@book{godsil2017algebraic,
  title={Algebraic combinatorics},
  author={Godsil, Chris},
  year={2017},
  publisher={Routledge}
}

@article{Adin1987,
  title={The conjugacy character of S n tends to be regular},
  author={Adin, Ron M and Frumkin, Avital},
  journal={Israel Journal of Mathematics},
  volume={59},
  pages={234--240},
  year={1987},
  publisher={Springer}
}

@article{fulmansym,
  title={Stein’s method and Plancherel measure of the symmetric group},
  author={Fulman, Jason},
  journal={Transactions of the American mathematical society},
  volume={357},
  number={2},
  pages={555--570},
  year={2005}
}

@article{Hora1998,
  title={Central limit theorem for the adjacency operators on the infinite symmetric group},
  author={Hora, Akihito},
  journal={Communications in mathematical physics},
  volume={195},
  pages={405--416},
  year={1998},
  publisher={Springer}
}

@article{Ram1997,
  title={Applications of the Frobenius formulas for the characters of the symmetric group and the Hecke algebras of type A},
  author={Ram, Arun and Remmel, Jeffrey B},
  journal={Journal of Algebraic Combinatorics},
  volume={6},
  pages={59--87},
  year={1997},
  publisher={Springer}
}

@inproceedings{olshanski,
  title={Kerov’s central limit theorem for the Plancherel measure on Young diagrams},
  author={Ivanov, Vladimir and Olshanski, Grigori},
  booktitle={Symmetric functions 2001: surveys of developments and perspectives},
  pages={93--151},
  year={2002},
  organization={Springer}
}

@ARTICLE{Kerov93gaussianlimit,
    author = {S. Kerov},
    title = {Gaussian Limit For The Plancherel Measure Of The Symmetric Group},
    journal = {C. R. Acad. Sci. Paris},
    year = {1993},
    volume = {316},
    pages = {303--308}
}

@article{diaconis111,
  title={Generating a random permutation with random transpositions},
  author={Diaconis, Persi and Shahshahani, Mehrdad},
  journal={Zeitschrift f{\"u}r Wahrscheinlichkeitstheorie und verwandte Gebiete},
  volume={57},
  number={2},
  pages={159--179},
  year={1981},
  publisher={Springer}
}

@book{stanley,
author = {Stanley, Richard P.},
title = {Enumerative Combinatorics: Volume 1},
year = {2011},
isbn = {1107602629},
publisher = {Cambridge University Press},
address = {USA},
edition = {2nd}
}

@misc{mellot,
      title={Gaussian concentration of the q-characters of the Hecke algebras of type A}, 
      author={Pierre-Loïc Méliot},
      year={2010},
      eprint={1009.4288},
      archivePrefix={arXiv},
      primaryClass={math.RT},
      url={https://arxiv.org/abs/1009.4288}
}
\end{document}